\documentclass{amsart}
\usepackage{palatino,amsmath,amsthm}
\usepackage{mathtools,thmtools}
\usepackage{hyperref,cleveref}
\usepackage{tabularx}
\usepackage{amssymb}
\usepackage{tikz-cd}
\usepackage{color}
\usepackage{fullpage}
\usepackage{hyperref}
\hypersetup{colorlinks, citecolor=red, filecolor=black, linkcolor=blue, urlcolor=blue, pdfpagemode=FullScreen}

\numberwithin{equation}{section}
\newtheorem{theorem}{Theorem}[section]
\newtheorem{proposition}[theorem]{Proposition}

\newtheorem{corollary}[theorem]{Corollary}

\theoremstyle{definition}
\newtheorem{definition}[theorem]{Definition}
\newtheorem{example}[theorem]{Example}
\newtheorem{remark}[theorem]{Remark}

\DeclareMathOperator{\Col}{\mathsf{PCol}}
\DeclareMathOperator{\type}{\mathsf{type}}
\DeclareMathOperator{\bitype}{\mathsf{btype}}
\DeclareMathOperator{\wtype}{\mathsf{wtype}}
\DeclareMathOperator{\wt}{\mathsf{wt}}
\DeclareMathOperator{\Mac}{\mathsf{Mac}}

\newcommand{\anti}{\mathsf{S}} 
\newcommand{\partn}{\vdash}
\newcommand{\sm}{\setminus}
\newcommand{\0}{\emptyset}

\newcommand{\littletaller}{\mathchoice{\vphantom{\big|}}{}{}{}}
\newcommand\restr[2]{{
  \left.\kern-\nulldelimiterspace 
  #1 
  \littletaller 
  \right|_{#2} 
  }}

\newcommand{\defterm}[1]{\textbf{#1}} 
\newcommand{\bLambda}{\boldsymbol{\Lambda}}
\newcommand{\blambda}{\boldsymbol{\lambda}}
\newcommand{\bOmega}{\boldsymbol{\Omega}}

\newcommand{\NN}{\mathbb{N}}	\newcommand{\Nn}{\mathbb{N}}
\newcommand{\PP}{\mathbb{P}}	\newcommand{\Pp}{\mathbb{P}}
\newcommand{\CC}{\mathbb{C}}	\newcommand{\Cc}{\mathbb{C}}

\renewcommand{\AA}{\mathcal{A}}
\newcommand{\Sym}{\mathfrak{S}}

\newcommand{\CSF}{\mathbf{X}}
\newcommand{\CMF}{\tilde{\mathbf{X}}}
\newcommand{\GDP}{\mathbf{G}}
\newcommand{\EGDP}{\tilde{\mathbf{G}}}
\newcommand{\ext}{\text{ext}}
\newcommand{\interior}{\text{int}}

\author[Martin]{Jeremy L. Martin}
\address[Jeremy L.~Martin]{Department of Mathematics, University of Kansas, Lawrence, KS 66045}
\email{\textcolor{blue}{\href{mailto:jlmartin@ku.edu}{jlmartin@ku.edu}}} 

\author[Trist]{May B. Trist}
\address[May B.~Trist]{Department of Mathematics, University of Kansas, Lawrence, KS 66045}
\email{\textcolor{blue}{\href{mailto:may.b.trist@gmail.com}{may.b.trist@gmail.com}}} 

\title{Chromatic MacMahon symmetric functions of graphs}
\date{\today}
\keywords{Chromatic symmetric function, generalized degree sequence, weighted graph, MacMahon symmetric function, Crew's conjecture}
\subjclass[2020]{%
05C05, 
05C22, 
05E05, 
16T30} 
\begin{document}

\begin{abstract}
A MacMahon symmetric function is an invariant of the diagonal action of the symmetric group on power series in multiple alphabets of variables.  We introduce an analogue of the chromatic symmetric function for vertex-weighted graphs, taking values in the MacMahon symmetric functions on two sets of variables, recording information about both cardinalities and weights of vertex sets.  We prove that the chromatic symmetric MacMahon function of a tree determines the generating function for its vertex subsets by cardinality, weight, and the numbers of internal and external edges.  This result generalizes the one for the unweighted case, first conjectured by Crew and proved independently by Aliste-Prieto--Martin--Wagner--Zamora and Liu--Tang.
\end{abstract}

\maketitle

\section{Background}

The \defterm{chromatic symmetric function} (or \defterm{CSF}) of a (finite, simple, undirected) graph $G=(V,E)$ is
\[
\CSF_G=\sum_{\kappa\in\Col(G)} \prod_{v\in V} x_{\kappa(v)}
\]
where $\Col(G)$ is the set of proper colorings of $G$, taking values in the positive integers~$\Pp$, and $x_1,x_2,\dots$ are commuting indeterminates.  The CSF was introduced in the context of knot theory by Chmutov, Duzhin and Lando~\cite{CDL} and in combinatorics by Stanley~\cite{S}; another important early paper connecting the two points of view is Noble and Welsh \cite{NW}. Stanley posed the problem of whether the CSF distinguishes trees up to isomorphism.  This problem remains open and is considered very difficult.  One approach is to study what other graph invariants can be recovered from the CSF of a tree \cite{MMW,CL,CL2,WYZ,AJMWZ,LT}.  In particular, Crew \cite{CL,CL2} conjectured that the CSF of a tree determines its \defterm{generalized degree polynomial (GDP)}, which is defined as
\[\GDP_G=\sum_{A\subseteq V}x^{|A|}y^{\ext(A)}z^{\interior(A)}\]
where $\ext(A)$ (resp., $\interior(A)$) is the number of edges of $E$ with one endpoint (resp., two endpoints) in $A$.  Crew's conjecture was proven by
Aliste-Prieto et al.~\cite[Thm.~6]{AJMWZ}, who gave an explicit linear transformation mapping the CSF to the GDP, and independently by Liu and Tang \cite[Prop.~2.4]{LT}, using Hopf algebra methods.

The CSF may be generalized to weighted graphs.  Let $G=(V,E)$ be a graph equipped with a weight function $\wt:V\to\Pp$.  The \defterm{weighted chromatic symmetric function} (or \defterm{wCSF}) of $G$ is
\[
\CSF_G=\sum_{\kappa\in\Col(G)} \prod_{v\in V} x_{\kappa(v)}^{\wt(v)}.
\]
This invariant was introduced by Crew and Spirkl \cite{CS} (although the idea of chromatic invariants of weighted graphs can be traced back to \cite{NW}), who showed that it admits a deletion/contraction recurrence, unlike the unweighted version.  For this reason, the wCSF has proven useful in attacking Stanley's tree uniqueness problem; see, e.g., \cite{APdMOZ}.

Aliste-Prieto asked the authors whether an analogue of Crew's conjecture holds for weighted graphs.  The most elementary way to adapt the GDP to the weighted setting is to replace $x^{|A|}$ in the definition by $x^{\wt(A)}$, where $\wt(A)=\sum_{v\in A}\wt(v)$.  In fact, this polynomial is \textit{not} determined by the wCSF.  We give a counterexample below in Figure~\ref{fig:weighted-Crew-fails} by adapting a construction of Loebl and Sereni \cite{LS}.

The proof of Crew's conjecture in \cite{AJMWZ} can \textit{almost} be adapted to the weighted setting, but the topological information provided by counting vertices appears to be indispensable.  This observation suggests expanding the definitions of both the CSF and GDP to keep track of both number of vertices and weight data.  Accordingly, we define a new coloring enumerator, the \defterm{chromatic MacMahon symmetric function}, by
\[
\CMF_G=\sum_{\kappa\in\Col(G)} \prod_{v\in V} x_{\kappa(v)} y_{\kappa(v)}^{\wt(v)}
\]
in two alphabets of commuting variables.  This power series is a \defterm{MacMahon symmetric function}: it is invariant under the \textit{diagonal} action of the symmetric group, acting simultaneously on each alphabet.  Power series of this form were introduced (under the name ``symmetric functions of several systems of quantities'') by MacMahon \cite[Sec.~XI]{MacMahon}, and studied more recently in \cite{R,RRS}.  
Similarly, we define the \defterm{extended generalized degree polynomial} (or \defterm{EGDP}) of a weighted graph $(G,\wt)$ as the polynomial
\[\EGDP_{G} = \EGDP_{G}(w,x,y,z) = \sum_{A\subseteq V}w^{\ext(A)}x^{|A|}y^{\wt(A)}z^{\interior(A)}.\]

We can now state the main result of the article, which generalizes Crew's conjecture.

\begin{restatable}{theorem}{MacMahonCrew}\label{thm:MacMahonCrew}
Let $(F,\wt)$ be a weighted forest.  Then $\EGDP_{F}$ is determined by $\CMF_F$.  
\end{restatable}

The paper is structured as follows.  Section~\ref{sec:basics} sets up basic definitions and tools for (weighted) graphs and chromatic symmetric functions.  Section~\ref{sec:MacMahon} concerns MacMahon symmetric functions, including the chromatic MacMahon symmetric function of a graph, its expansion
in the MacMahon power-sum basis, and the Hopf algebra structure of MacMahon symmetric functions.
Section~\ref{sec:ProofMain} contains two proofs of Theorem~\ref{thm:MacMahonCrew}, by adapting each of the arguments of \cite{LT} and \cite{AJMWZ} to the weighted setting.  Finally, in Section~\ref{sec:additional}, we observe that the theory of chromatic bases of symmetric functions~\cite{CvW} carries over well to MacMahon symmetric functions, and discuss an easy generalization of the theory to $\Pp^r$-weighted graphs.

We thank Jos\'e Aliste-Prieto for suggesting this line of research, and Ira Gessel and Mercedes Rosas for helpful references on MacMahon symmetric functions.

\section{Basic definitions and notation}\label{sec:basics}

The symbols $\PP$ and $\NN$ denote the positive integers and the nonnegative integers, respectively.  We write $[n]$ for the set $\{1,2,\dots,n\}$.
We assume familiarity with standard notions of graph theory; see, e.g., \cite{Diestel}.  All graphs are assumed to be finite, simple, and undirected.

We either write a graph as an ordered pair $G=(V,E)$, or use the notation $V(G)$ and $E(G)$ for its vertex and edge sets, as convenient.
The symbols $n(G),e(G),c(G)$ denote respectively the numbers of vertices, edges, and connected components of~$G$.
For $A\subseteq V(G)$, we write $E(A)$ for the set of edges of $G$ with both endpoints in~$A$.  The subgraph induced by $A$ is $\restr{G}{A}=(A,E(A))$.

A \defterm{weighted graph} $(G,\wt)=(V,E,\wt)$ is a graph $G=(V,E)$ together with a function $\wt\colon V\to\PP$.  The \defterm{total weight} of $G$ is $\wt(G)=\sum_{v_i\in V}\wt(v_i)$.

\begin{definition} \label{defn:typeG}
The \defterm{type} of $G$ is the partition $\type(G)\partn n(G)$ whose parts are the numbers of vertices of its connected components. 
For an edge set $S\subseteq E$, we define $\type(S)=\type(V,S)$.
Similarly, the \defterm{weighted type} of a weighted graph $(G,\wt)$ is the partition $\wtype(G)\partn\wt(G)$ whose parts are the total weights of its connected components, and for $S\subseteq E$, we define $\wtype(S)=\wtype(V,S,\wt)$.
\end{definition}

A \defterm{coloring} of a graph is a function $\kappa\colon G\to\PP$.  A coloring $\kappa$ is \defterm{proper} if $\kappa(v)\neq \kappa(w)$ whenever $vw\in E(G)$.  The set of all proper colorings of $G$ is denoted by $\Col(G)$.

\begin{definition} \label{defn:CSF} \cite{S}
The \defterm{chromatic symmetric function} (or \defterm{CSF}) of a graph $G$ is the power series
\[\CSF_G=\sum_{\kappa\in\Col(G)} \prod_{v\in V} x_{\kappa(v)}.\]
\end{definition}

The CSF is a symmetric function in commuting variables $x_1,x_2,\dots$.  A standard reference on symmetric functions is \cite[Chap.~7]{EC2}.
The power-sum expansion of the CSF is
\begin{equation}\label{csfexp}
\CSF_G=\sum_{S\subseteq E(G)} (-1)^{|S|} p_{\type(S)}
\end{equation}
\cite[Thm. 2.5]{S}.  When $G=F$ is a forest (and not otherwise), there is no cancellation in this expression, and the formula may be rewritten as
\begin{equation}\label{csfexp-tree}
\CSF_F=\sum_{\lambda\vdash n} \beta_{\lambda}(F)(-1)^{n-\ell(\lambda)} p_{\lambda}
\end{equation}
where $\beta_{\lambda}(F)=|\{S\subseteq E(F)\colon \type(S)=\lambda\}|$ \cite[Cor.~2.8]{S}.

\begin{definition} \label{defn:GDP} \cite[Sec.~4.3]{CL}
Let $G$ be a graph and $A\subseteq V(G)$.  Say that an edge is \defterm{internal} to $A$ if it has both endpoints in $A$, and \defterm{external} to $A$ if it has exactly one endpoint in $A$.  The numbers of internal and external edges are denoted by $\interior(A)$ and $\ext(A)$ respectively.  The \defterm{generalized degree polynomial} (or \defterm{GDP}) of $G$ is the polynomial
\[\GDP_G=\GDP_G(x,y,z)\sum_{A\subseteq V(G)}x^{|A|}y^{\ext(A)}z^{\interior(A)}.\]
\end{definition}

These invariants naturally generalize to weighted graphs, as studied in \cite{CS,CL}.

\begin{definition} \label{defn:wCSF} \cite{CS}
The \defterm{weighted chromatic symmetric function} (or \defterm{wCSF}) of a weighted graph $(G,\wt)$ is the power series
\[\CSF_{G,\wt}=\sum_{\kappa\in\Col(G)} \prod_{v\in V} x_{\kappa(v)}^{\wt(v)}.\]
\end{definition}
The wCSF is equivalent to the \defterm{$W$-polynomial} introduced in \cite{NW}; see also \cite{LS}.

Stanley's proof of~\eqref{csfexp} carries over easily to the weighted setting, so that~\eqref{csfexp} and~\eqref{csfexp-tree} have the analogues
\begin{align}
\CSF_{G,\wt} &= \sum_{S\subseteq E}(-1)^{|S|} p_{\wtype(S)},\\
\intertext{or, when $F$ is a forest,}
\CSF_{F,\wt} &= \sum_{\lambda\vdash n} \beta_{\lambda,\wt}(F)(-1)^{n-\ell(\lambda)} p_{\lambda},
\end{align}
where $\beta_{\lambda,\wt}(F)=|\{S\subseteq E(F)\colon \wtype(S)=\lambda\}|$.

\begin{definition} \label{defn:wGDP}
The \defterm{weighted generalized degree polynomial} (or \defterm{wGDP}) of a weighted graph $(G,\wt)$ is the polynomial
\[
\GDP_{G,\wt}=\GDP_{G,\wt}(x,y,z)=\sum_{A\subseteq V(G)}x^{\wt(A)}y^{\ext(A)}z^{\interior(A)}=\sum_{a,b,c}g_T(a,b,c)x^ay^bz^c.
\]
where $g_T(a,b,c)=|\{A\subseteq V\colon\ \wt(A)=a,\ \ext(A)=b,\ \interior(A)=c\}|$.
\end{definition}

Crew~\cite{CL2}  conjectured that the CSF of a tree determines its GDP.  This conjecture was proven by Aliste-Prieto, Martin, Wagner and Zamora \cite[Thm.~6]{AJMWZ} and independently by Liu and Tang \cite[Prop. 2.4]{LT} using different methods.  On the other hand, the weighted analogue of Crew's conjecture is false.  Consider the two weighted 5-vertex paths $T_1,T_2$ shown in \Cref{fig:weighted-Crew-fails}, where the numbers indicate weights.

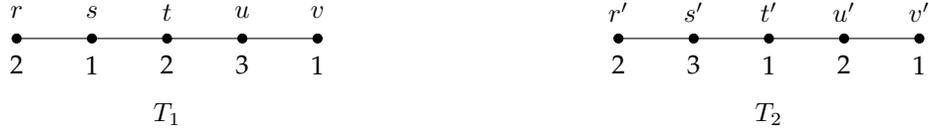
\begin{figure}[ht]
\begin{center}
\begin{tikzpicture}
\draw (0,0)--(4,0);
\foreach \x/\l/\lu in {0/2/r, 1/1/s, 2/2/t, 3/3/u, 4/1/v} { \draw[fill=black] (\x,0) circle (.06); \node at (\x,-.35) {\l}; \node at (\x,.35) {$\lu$}; }
\node at (2,-1) {$T_1$};
\begin{scope}[shift={(8,0)}]
\draw (0,0)--(4,0);
\foreach \x/\l\lu in {0/2/r', 1/3/s', 2/1/t', 3/2/u', 4/1/v'} { \draw[fill=black] (\x,0) circle (.06); \node at (\x,-.35) {\l}; \node at (\x,.35) {$\lu$};}
\node at (2,-1) {$T_2$};
\end{scope}
\end{tikzpicture}
\end{center}
\caption{Two weighted trees with the same wCSF but different wGDPs.}
\label{fig:weighted-Crew-fails}
\end{figure}

It was observed in \cite[p.~5]{LS}, $T_1$ and $T_2$ have the same $W$-polynomial, hence the same wCSF.  On the other hand, let us count the vertex sets whose total weight, external edge count, and internal edge count are 4,3,0 respectively.  $T_1$ has exactly one such set, namely $\{r,t\}$, and $T_2$ has two such sets, namely $\{r',u'\}$ and $\{s',v'\}$.  Therefore, the coefficients of $x^4y^3z^0$ are different in $\GDP_{T_1}$ and $\GDP_{T_2}$.

\section{MacMahon symmetric functions} \label{sec:MacMahon}

Let $m\in\NN$ and let $\AA=\{x_{j,k}\colon j\in\NN,\ k\in[m]\}$ be a family of commuting indeterminates.  For each $k\in[m]$, the subset $\AA_k=\{x_{j,k}\colon j\in\NN\}$ is the $k$th \defterm{alphabet}.
Let $\CC[[\AA]]$ denote the ring of formal power series in the variables $x_{j,k}$, with coefficients in $\CC$ (here and later one could replace $\CC$ with any field of characteristic~0).  This ring is $\Nn^m$-multigraded: for each $\mathbf{u}=(u_1,\dots,u_m)\in\Nn^m$, the $\mathbf{u}$-multigraded piece is spanned by monomials whose total degree in $\AA_k$ is $u_k$ for each $k\in[m]$.
Let $\Sym_\infty$ denote the group of permutations of $\NN$.

\begin{definition}\label{defn:MacMahon}
The \defterm{diagonal action} of $\Sym_\infty$ on $\AA$ is defined by
\[\sigma(x_{j,k})=x_{\sigma(j),k}.\]
The \defterm{MacMahon symmetric functions} are the invariants of the diagonal action.  They form a $\Nn^m$-multigraded subalgebra of $\Cc[[\AA]]$ denoted by $\Mac^m$.
\end{definition}

The ring $\Mac^1$ is just the familiar ring $\Lambda$ of symmetric functions.  For general $m$, the ring $\Mac^m$ has analogues of the monomial, power-sum, elementary, and complete homogeneous bases of $\Lambda$, as described in~\cite{R}, all of which are indexed by objects called \defterm{vector partitions}, which generalize integer partitions.

\begin{definition} \label{defn:vector-partition}
Let $\mathbf{u}=(u_1,\dots,u_m)\in\Pp^m$.
A \defterm{vector partition} of $\mathbf{u}$ is an unordered list $\bLambda=(\blambda^{(1)},\dots,\blambda^{(\ell)})$ of vectors $\blambda^{(i)}\in\Nn^m\sm\{(0,0\dots,0)\}$ such that $\blambda^{(1)}+\cdots+\blambda^{(\ell)}=\mathbf{u}$.  
For short, we write $\bLambda\vdash \mathbf{u}$.
The vectors $\blambda^{(i)}$ are the \defterm{parts} of $\bLambda$.  The number of parts is called its \defterm{length}, written $\ell(\bLambda)$, and the number $m$ is its \defterm{width}.
Note that a vector partition of width~1 is simply an integer partition.
\end{definition}

\begin{definition} \label{defn:powersum-MacMahon}
Let $\blambda=(\lambda_1,\dots,\lambda_m)\in\Nn^m$. The corresponding \defterm{power-sum MacMahon symmetric function} is 
\[p_{\blambda}=\sum_{j=1}^{\infty}\prod_{k=1}^{m} (x_{j,k})^{\lambda_{k}}.\]
For a vector partition $\bLambda=(\blambda^{(1)},\dots,\blambda^{(\ell)})$ of width~$m$, we define the power-sum MacMahon symmetric function $p_{\bLambda}$ by
\[p_{\bLambda} = p_{\blambda^{(1)}}\cdots p_{\blambda^{(\ell)}}.\]
For each $\mathbf{u}\in\Nn^m$, the set $\{p_{\bLambda}\mid \bLambda\partn\mathbf{u}\}$ is a vector space basis for the $\mathbf{u}$-multigraded piece of $\Mac^m$.
\end{definition}

Henceforth, we focus on the ring $\Mac^2$.  We simplify notation by setting $x_j=x_{j,1}$ and $y_j=x_{j,2}$.

\begin{example}
The power-sum basis for the graded piece of $\Mac^2$ with multidegree $\mathbf{u}=(2,1)$ consists of the following elements:
\begin{align*}
p_{((2,1))}  &= x_1^2y_1+x_2^2y_2+x_3^2y_3+\cdots\\
p_{((2,0),(0,1))}  &= (x_1^2+x_2^2+x_3^2+\cdots)(y_1+y_2+y_3+\cdots)\\
p_{((1,1),(1,0))}  &= (x_1y_1+x_2y_2+x_3y_3+\cdots)(x_1+x_2+x_3+\cdots)\\
p_{((1,0),(1,0),(0,1))}  &= (x_1+x_2+x_3+\cdots)^2(y_1+y_2+y_3+\cdots)
\end{align*}
\end{example}

\begin{definition} \label{defn:bitype}
Let $(G,\wt)$ be a weighted graph with connected components $C_1,\dots,C_k$, and let $\mathbf{u}=(n(G),\wt(G))$.  The \defterm{bitype} of $G$ is the vector partition
\[\bitype(G)=\big( (n(C_1),\wt(C_1)), \; \dots, \; (n(C_k),\wt(C_k)) \big) \partn \mathbf{u}.\]
For $S\subseteq E$, we set $\bitype(S)=\bitype(V,S,\wt)$. Similarly, if $A\subseteq V$, we set $\bitype(A)=\bitype(A,E(A),\restr{\wt}{A})$, where $E(A)$ is the set of edges internal to $A$.
\end{definition}

Not every vector partition can occur as a bitype of a graph.  Specifically, each part must be a vector in $\Pp^2$ (not merely in $\Nn^2\sm\{(0,0)\}$).

\begin{definition} \label{defn:CMF}
The \defterm{chromatic MacMahon symmetric function} (or \defterm{CMF}) of a weighted graph $G=(V,E,\wt)$ is the power series
\[\CMF_G=\sum_{\kappa\in\Col(G)} \prod_{v\in V} x_{\kappa(v)} y_{\kappa(v)}^{\wt(v)}.\]
\end{definition}

As a note, this chromatic MacMahon symmetric function is unrelated to the one defined by Rosas~\cite[Defn.~9]{R}.

\begin{remark}\label{remark:determine}
The chromatic MacMahon symmetric function of a weighted graph~$G$ determines $\CSF_{G,\wt}$ by setting $x_i=1$ for every $i$, and determines $\CSF_{G}$ by setting $y_i=1$ for every $i$. On the other hand, $\CSF_{G,\wt}$ and $\CSF_G$ do not together determine $\CMF_G$.  For instance, let $T_1,T_2$ be the two weighted trees shown in \Cref{fig:weighted-Crew-fails}, which have the same CSF and weighted CSF.  On the other hand, if we set $x_i=y_i=0$ for all $i\geq 3$ (i.e., we consider only colorings $\kappa\colon V\to\{1,2\}$), then the chromatic MacMahon symmetric functions of $T_1$ and $T_2$ become respectively
\[x_1^3y_1^5x_2^2y_2^4+x_1^2y_1^4x_2^3y_2^5 \quad \text{and} \quad x_1^3y_1^4x_2^2y_2^5+x_1^2y_1^5x_2^3y_2^4,\]
so $\CMF_{T_1}\neq\CMF_{T_2}$.
\end{remark}

The chromatic MacMahon symmetric function admits a power-sum expansion analogous to~\eqref{csfexp}.  The proof of that result in \cite{S} carries over to the setting of MacMahon symmetric functions, as we now show.

\begin{proposition} \label{power-CMF}
Let $(G,\wt)$ be a weighted graph. Then
\[\CMF_G=\sum_{S\subseteq E(G)}(-1)^{|S|}p_{\bitype(S)}.\]
\end{proposition}

\begin{proof}
Fix $S\subseteq E$, and suppose that $G|_S=(V,S,\wt)$ has connected components $C_1,\dots,C_\ell$, of sizes $n_1,\dots,n_\ell$ and weights $w_1,\dots,w_\ell$.  Then $\bitype(S)=(n_1,w_1)\cdots(n_{\ell},w_{\ell})$ and
\begin{align*}
p_{\bitype(S)} &= \prod_{i=1}^\ell p_{(n_i,w_i)} \\
&= \sum_{(k_1,\dots,k_\ell)\in\mathbb{N}^\ell} x_{k_1}^{n_1}y_{k_1}^{w_1}\cdots x_{k_{\ell}}^{n_{\ell}}y_{k_{\ell}}^{w_{\ell}}\\
&= \sum_{(k_1,\dots,k_\ell)\in\mathbb{N}^\ell} x_{k_1}^{|C_1|}y_{k_1}^{\sum_{v\in C_1}\wt(v)}\cdots x_{k_{\ell}}^{|C_{\ell}|}y_{k_{\ell}}^{\sum_{v\in C_{\ell}}\wt(v)}\\
&= \sum_{\kappa\in K_S(G)} \prod_{v\in V} x_{\kappa(v)}y_{\kappa(v)}^{\wt(v)}
\end{align*}
where $K_S(G)$ is the set of all colorings $\kappa\colon V\to\NN$ that are monochromatic on every $C_i$ (specifically, assigning color $k_i$ to all vertices of $C_i$).
Multiplying by $(-1)^{|S|}$ and summing over all $S$, we obtain
\begin{align*}
\sum_{S\subseteq E(G)}(-1)^{|S|}p_{\bitype(S)}
&= \sum_{S\subseteq E(G)}(-1)^{|S|}\sum_{\kappa\in K_S(G)} \prod_{v\in V} x_{\kappa(v)}y_{\kappa(v)}^{\wt(v)}\\
&= \sum_{\kappa:V\to\NN} \prod_{v\in V} x_{\kappa(v)}y_{\kappa(v)}^{\wt(v)} \left(\sum_{S\subseteq E_\kappa(G)}(-1)^{|S|}\right)
\end{align*}
where $E_\kappa(G)=\{uv\in E(G)\mid \kappa(u)=\kappa(v)\}$.  The parenthesized sum is 1 if $E_\kappa(G)=\0$ (i.e., when $\kappa$ is a proper coloring) and 0 otherwise.  Therefore,
\[
\CMF_G=\sum_{\kappa\in\Col(G)} \prod_{v\in V} x_{\kappa(v)}y_{\kappa(v)}^{\wt(v)}=\sum_{S\subseteq E(G)}(-1)^{|S|}p_{\bitype(S)}.\qedhere
\]
\end{proof}

When $F$ is a forest, the following analogue of~\eqref{csfexp-tree} follows from grouping the terms in \Cref{power-CMF} by the vector partition $\bitype(S)$ and observing that $|S|=n-\ell(\bitype(S))$, so no cancellation occurs.

\begin{corollary} \label{power-CMF-forest}
Let $(F,\wt)$ be a weighted forest with $n=n(F)$ and $w=\wt(F)$.  Then
\[\CMF_F=\sum_{\bLambda\vdash(n,w)}\beta_{\bLambda}(F)(-1)^{n-\ell(\bLambda)}p_{\bLambda}\]
where
\[\beta_{\bLambda}(F)=|\{S\subseteq E(F)\colon \bitype(S)=\bLambda\}|.\] 
\end{corollary}

Observe that $\beta_{\bLambda}(F)=0$ if $\bLambda$ has any parts not in $\Pp^2$.  Moreover, for each subset $A\subseteq E(F)$, the subgraph $(V(F),A)$ has $n-|A|$ components, so its bitype has length $n-|A|$.  Therefore, we obtain the useful equation
\begin{equation} \label{sum-b}
\sum_{\substack{\bLambda\vdash (n(F),\wt(F))\\ \ell(\bLambda)=\ell}} \beta_{\bLambda}=\binom{e(F)}{n-\ell}.
\end{equation}

As observed by Rosas, Rota and Stein \cite{RRS}, the MacMahon symmetric functions admit a Hopf algebra structure. We very briefly sketch the definition of a combinatorial Hopf algebra; for more details, see, e.g., \cite{GR}.

\begin{definition}\cite[Chap.~1]{GR}\label{defn:bialgebra-Hopf}
A \defterm{bialgebra} is a vector space $A$ over a field (say $\Cc$) endowed with a \defterm{product} $m\colon  A\otimes A\to A$, a \defterm{coproduct} $\Delta\colon  A\to A\otimes A$, a \defterm{unit} $u\colon\Cc\to A$, and a \defterm{counit} $\varepsilon\to\Cc$, satisfying a variety of axioms, of which the three most important are \defterm{associativity}, \defterm{coassociativity}, and \defterm{compatibility}, given by the diagrams shown below.
\[\begin{array}{ccccc}
\begin{tikzcd}
A\otimes A\otimes A \arrow{r}{\mu} \arrow[swap]{d}{\mu} & A\otimes A \arrow{d}{\mu\otimes I} \\
A\otimes A \arrow{r}{I\otimes\mu}& A
\end{tikzcd}
&&
\begin{tikzcd}
A \arrow{r}{\Delta} \arrow[swap]{d}{\Delta} & A\otimes A \arrow{d}{\Delta\otimes I} \\
A\otimes A \arrow{r}{I\otimes\Delta}& A\otimes A\otimes A
\end{tikzcd}
&&
   \begin{tikzcd}
A\otimes A \arrow{r}{m} \arrow[swap]{d}{\Delta\otimes\Delta} & A \arrow{d}{\Delta} \\
A\otimes A\otimes A\otimes A\quad  \arrow{r}{m_{1,3}\otimes m_{2,4}}&\quad A\otimes A
\end{tikzcd}
\\
\text{associativity} && \text{coassociativity} && \text{compatibility}
\end{array}\]
where $m_{i,j}$ denotes multiplication of the $i$th and $j$th coordinates.

The bialgebra $A$ is \defterm{graded} if there is a vector space decomposition
$A=\bigoplus_{n\geq 0} A_n$ such that
\[\forall n,k\colon\ m(A_k\otimes A_{n-k})\subseteq A_n \quad\text{and}\quad \forall n\colon \Delta A_n\subseteq \bigoplus_{k\leq n}A_k\otimes A_{n-k}.\]
A graded bialgebra $A$ is \defterm{connected} if $\dim_\Cc A_0=1$.
Every graded connected bialgebra has a unique structure as a \defterm{Hopf algebra} \cite[Prop.~1.4.16]{GR}; that is, there is an \defterm{antipode map} $\anti\colon A\to A$, defined recursively as follows: $\anti$ is the identity on $A_0$, and, for all $x\in A_k$ with $k>0$, we have
\begin{equation}\label{antipode}
\sum S(x_1) x_2 = 0
\end{equation}
in Sweedler notation~\cite[p.8]{GR}.
\end{definition}

The ring $\Mac^m$ of MacMahon symmetric functions is evidently a graded connected $\Cc$-algebra.  In fact, we prove that it is a bialgebra.  For a vector partition $\bLambda=(\blambda^{(1)},\dots,\blambda^{(\ell)})$ and $J\subseteq[\ell]$, let $\restr{\bLambda}{J}$ be the vector partition with parts $\blambda^{(i)}$ for $i\in J$, and let $\bar J=[\ell]\sm J$.

\begin{proposition} \label{prop:MacMahon-coproduct}
The ring $\Mac^m$ is a graded connected Hopf algebra, with (i) product defined on the power sum basis by
\begin{equation}\label{product}
p_{\bOmega}\, p_{\bLambda} = p_{\bLambda\bOmega}
\end{equation}
where $\bLambda\bOmega$ is the vector partition obtained by concatenating $\bOmega$ with~$\bLambda$, and (ii) coproduct defined by
\begin{equation}\label{coproduct}
\Delta(p_{\bLambda})=\sum_{J\subseteq [\ell(\bLambda)]}p_{\restr{\bLambda}{J}}\otimes p_{\restr{\bLambda}{\overline{J}}}.
\end{equation}
\end{proposition}

\begin{remark}
It was proven by Rosas, Rota and Stein~\cite{RRS} that $\Mac^m$ is a Hopf algebra.  The content of Proposition~\ref{prop:MacMahon-coproduct} is the coproduct formula in the power-sum basis, which to our knowledge does not appear explicitly in the literature.  It generalizes the well-known formula for power-sum symmetric functions (equation~\eqref{multiply-p-symfn}, below), which is just the case $m=1$, and it naturally resembles the formula for power-sum \textit{noncommutative} symmetric functions given explicitly by Lauve and Mastnak~\cite[eqn.~(3)]{LM}.  Indeed, it is possible to verify~\eqref{coproduct} by applying the projection map of Rosas~\cite[Defn.~1]{R} to the Lauve--Mastnak formula.  Instead, we give a self-contained proof without using noncommutative symmetric functions.
\end{remark}

\begin{proof}
The product formula is immediate from Definition~\ref{defn:powersum-MacMahon}, so $\Mac^m$ is a graded connected ring.  It remains to show that the coproduct~\eqref{coproduct} satisfies coassociativity and compatibility.
For coassociativity, it is routine to show that
\[
(I\otimes\Delta)(\Delta(p_{\bLambda}))
=\sum_{J_1\sqcup J_2\sqcup J_3=[\ell(\bLambda)]}p_{\restr{\bLambda}{J_1}}\otimes p_{\restr{\bLambda}{J_2}}\otimes p_{\restr{\bLambda}{J_3}}\\
=\Delta\otimes I\left(\sum_{J\subseteq [\ell(\bLambda)]}p_{\restr{\bLambda}{J}}\otimes p_{\restr{\bLambda}{\overline{J}}}\right)
\]
For compatibility, observe that
\begin{align*}
m_{1,3}\otimes m_{2,4}\left(\Delta\otimes\Delta\left(p_{\bLambda}\otimes p_{\bOmega}\right)\right)
&=\sum_{\substack{J\subseteq [\ell(\bLambda)]\\J'\subseteq[\ell(\bOmega)]}}(p_{\restr{\bLambda}{J}}\cdot p_{\restr{\bOmega}{J'}})\otimes (p_{\restr{\bLambda}{\overline{J}}}\cdot p_{\restr{\bOmega}{\overline{J'}}})\\
&=\sum_{K\subseteq [\ell(\bLambda\bOmega)]}p_{\restr{\bLambda\bOmega}{K}}\otimes p_{\restr{\bLambda\bOmega}{\overline{K}}}\\
&=\Delta(p_{\bLambda\bOmega}) = \Delta(p_{\bLambda}\cdot p_{\bOmega}) = \Delta(m(p_{\bLambda}\otimes p_{\bOmega})).\qedhere
\end{align*}
\end{proof}

The Hopf algebra $\Mac^m$ is evidently commutative and cocommutative.  The special case $\Mac^1$ is just the standard Hopf algebra of symmetric functions, with coproduct given on the power-sum symmetric functions by
\[\Delta(p_n)=1\otimes p_n +p_n\otimes 1\]
\cite[Prop.~2.3.6]{GR} and
\begin{equation} \label{multiply-p-symfn}
\Delta(p_{\lambda})=\sum_{J\subseteq[\ell(\lambda)]}p_{\restr{\lambda}{J}}\otimes p_{\restr{\lambda}{\overline{J}}}.
\end{equation}

When $\bLambda$ consists of a single vector $\blambda$, 
\Cref{prop:MacMahon-coproduct} specializes to $\Delta(p_{\blambda})=p_{\blambda}\otimes 1+1\otimes p_{\blambda}$, so (when $\blambda\neq0$) the antipode formula~\eqref{antipode} yields $\anti(p_{\blambda})=-p_{\blambda}$.  Therefore, the antipode acts on the basis $\{p_{\bLambda}\}$ by
\begin{equation}
\anti(p_{\bLambda})=(-1)^{\ell(\bLambda)}p_{\bLambda},
\end{equation}
generalizing the well-known result for symmetric functions \cite[Prop.~2.4.1(i)]{GR}

The coproduct of a chromatic MacMahon symmetric symmetric function also has a particularly simple form.

\begin{proposition} \label{prop:coprod-CMF}
Let $(G,\wt)$ be a weighted graph.  Then
\[\Delta(\CMF_G)=\sum_{A\subseteq V(G)}\CMF_{\restr{G}{A}}\otimes \CMF_{\restr{G}{\bar{A}}}.\]
\end{proposition}

\begin{proof}
For an edge set $S\subseteq E(G)$, let $K(S)$ denote the set of connected components of $(V,S)$.  For convenience, we identify each subset $J\subseteq K(S)$ with the disjoint union of its elements.  In particular, $V(\bar J)=V(G)\sm V(J)$ and $E(\bar J)=S\sm E(\bar J)$.
By \Cref{power-CMF} and the definition of coproduct,
\begin{align*}
\Delta(\CMF_G)&=\sum_{S\subseteq E(G)}(-1)^{|S|}\Delta(p_{\bitype(S)})\\
&=\sum_{S\subseteq E(G)}(-1)^{|S|}\sum_{J\subseteq K(S)}p_{\bitype(J)}\otimes p_{\bitype(\bar{J})}\\
&=\sum_{S\subseteq E(G)}\sum_{J\subseteq K(S)}(-1)^{|E(J)|}p_{\bitype(J)}\otimes (-1)^{|E(\bar{J})|}p_{\bitype(\bar{J})}\\
&=\sum_{A\subseteq V(G)}\sum_{S_1\subseteq E(A)}\sum_{S_2\subseteq E(\bar{A})}(-1)^{|S_1|}p_{\bitype(S_1)}\otimes (-1)^{|S_2|}p_{\bitype(S_2)}\\
&=\sum_{A\subseteq V(G)}\left(\sum_{S_1\subseteq E(A)}(-1)^{|S_1|}p_{\bitype(S_1)}\right)\otimes \left(\sum_{S_2\subseteq E(\bar{A})}(-1)^{|S_2|}p_{\bitype(S_2)}\right)\\
&=\sum_{A\subseteq V(G)}\CMF_{\restr{G}{A}}\otimes \CMF_{\restr{G}{\bar{A}}}.\qedhere
\end{align*}
\end{proof}

Given a Hopf algebra $A$ and two $\CC$-linear maps $f,g\colon A\to R$, where $R$ is a commutative $\CC$-algebra, the \defterm{convolution} of $f$ and $g$ is the function $f*g\colon A\to R$ given by
\[(f*g)(B)=\sum f(B_1)\cdot g(B_2)\]
where $\Delta(B)=\sum B_1\otimes B_2$ in Sweedler notation.  The following is an immediate consequence of \Cref{prop:coprod-CMF}.

\begin{corollary} \label{cor:convolution}
Let $f,g\colon\Mac^2\to R$ be linear maps and $(G,\wt)$ a weighted graph.  Then
\[(f*g)(\CMF_G)=\sum_{A\subseteq V(G)}f(\CMF_{\restr{G}{A}})\cdot g(\CMF_{\restr{G}{\bar{A}}}).\]
\end{corollary}

\section{Proof of the main theorem}\label{sec:ProofMain}

\begin{definition} \label{defn:EGDP}
The \defterm{extended generalized degree polynomial} (or \defterm{EGDP}) of a weighted graph $(G,\wt)$ is the polynomial
\[\EGDP_{G}
= \EGDP_{G}(w,x,y,z)
= \sum_{A\subseteq V}w^{\ext(A)}x^{|A|}y^{\wt(A)}z^{\interior(A)}
= \sum_{a,b,c,d}g_T(a,b,c,d)w^ax^by^cz^d\]
where $g_T(a,b,c,d)=|\{A\subseteq V:\ \ext(A)=a,\ |A|=b,\ \wt(A)=c,\ \interior(A)=d\}|$.
\end{definition}

Our main theorem is that Crew's conjecture holds in the more general setting of chromatic MacMahon symmetric functions of weighted forests.

\MacMahonCrew*  

Crew's original conjecture is obtained as the special case when all weights are~1.

We give two proofs of this theorem.  The first is modeled on Liu and Tang's Hopf-theoretic proof of Crew's conjecture \cite[Prop.~2.4]{LT}.

\begin{proof}[First proof of Theorem~\ref{thm:MacMahonCrew}]
We construct a $\mathbb{Q}$-linear map $\gamma\colon M^2\to \mathbb{Q}(w,x,y,z)$ with the property that
\begin{equation} \label{desired-gamma}
\gamma(\CMF_F)=y^{c(F)}\EGDP_F(w,x,y,z).
\end{equation}

We first show how to recover the numbers $n=n(F)$, $e=e(F)$, and $w=\wt(F)$ from $\CMF_F$.  (The corresponding statement in the unweighted case is \cite[Lemma~2.5]{LT}.)  Consider the $\mathbb{Q}$-linear map
\[\varphi_{t,u,v}\colon \Mac^2\to \mathbb{Q}[t,u,v]\]
defined on the power-sum basis by
\[\varphi_{t,u,v}(p_{\bLambda})=t^{n}(1-u)^{n-\ell(\bLambda)}v^{w} \quad \text{ for } \bLambda\partn(n,w).\]
By \Cref{power-CMF-forest},
\begin{align*}
\varphi_{t,u,v}(\CMF_F)&=\sum_{\bLambda\vdash (n,w)} \beta_{\bLambda}(F)(-1)^{n-\ell(\bLambda)} \varphi_{t,u,v}(p_{\bLambda})\\
&=\sum_{\bLambda\vdash (n,w)} \beta_{\bLambda}(F)(-1)^{n-\ell(\bLambda)} t^n (1-u)^{n-\ell(\bLambda)}v^{w}\\
&=t^n v^w \sum_{\bLambda\vdash(n,w)} \beta_{\bLambda}(F) (u-1)^{n-\ell(\bLambda)}\\
&=t^n v^w \sum_{\ell=1}^n \sum_{\substack{\bLambda\vdash (n,w)\\ \ell(\bLambda)=\ell}} \beta_{\bLambda}(F) (u-1)^{n-\ell}\\
&=t^n v^w \sum_{\ell=1}^n \binom{e}{n-\ell}(u-1)^{n-\ell}\\
&=t^n u^{e} v^w.
\end{align*}
where the second-to-last equality follows from~\eqref{sum-b}.

Now, we claim that the map $\gamma$ defined by
\[\gamma=\varphi_{wx,w^{-1}z,y}*\varphi_{w,w^{-1},1}\]
has the desired property~\eqref{desired-gamma}.  Indeed, by \Cref{cor:convolution},
\begin{align*}
(\varphi_{wx,w^{-1}z,y}*\varphi_{w,w^{-1},1})(\CMF_F)&=\sum_{A\subseteq V(F)}\varphi_{wx,w^{-1}z,y}(\CMF_{\restr{F}{A}})\cdot\varphi_{w,w^{-1},1}(\CMF_\restr{F}{\overline{A}})\\
&=\sum_{A\subseteq V(F)}(wx)^{\left|A\right|}(w^{-1}z)^{\interior(A)}(y)^{\wt(A)}(w)^{|\overline{A}|}(w^{-1})^{\interior(\overline{A})}\\
&=\sum_{A\subseteq V(F)}w^{n(F)-\interior(A)-\interior(\overline{A})}x^{\left|A\right|}y^{\wt(A)}z^{\interior(A)}\\
&=\sum_{A\subseteq V(F)}w^{\ext(A)+c(F)}x^{\left|A\right|}y^{\wt(A)}z^{\interior(A)}\\
&=w^{c(F)}\textbf{G}_F(w,x,y,z)
\end{align*}
and $c(F)=n(F)-e(F)$ can be recovered from $\CMF_F$ by the first calculation, completing the proof.
\end{proof}

The second proof of Theorem~\ref{thm:MacMahonCrew} uses the method of Aliste-Prieto et~al.~\cite[Thm.~6]{AJMWZ}: we express the coefficients of $\EGDP_{F}(w,x,y,z)$ as explicit linear combinations of the coefficients of $\CMF_F$.

We first need some notation.  For vector partitions $\bLambda,\bOmega\vdash (n,w)$, define 
\begin{align*}
\binom{\bLambda}{\bOmega} = \prod_{i=1}^{a}\prod_{j=1}^{d}\binom{m_{i,j}(\bLambda)}{m_{i,j}(\bOmega)}
\end{align*}
where $m_{i,j}(\bLambda)$ is the number of times the vector $(i,j)$ appears in $\bLambda$. Observe that if $G$ is a weighted graph and $\bLambda=\bitype(S)$ for some $S\subseteq E(G)$, then 
\[\binom{\bitype(S)}{\bOmega}=|\{A\subseteq V(G):\ \bitype(A)=\bOmega,\ S\subseteq E(A)\cup E(\overline{A})\}|.\]

We need the following simple combinatorial identity \cite[Lemma~5]{AJMWZ}:
for every set $P$ and every $q\in\NN$, we have
\begin{equation} \label{useful-lemma}
\sum_{P'\subseteq P} (-1)^{|P'|+q} \binom{|P'|}{q} = \begin{cases} 1 & \text{ if } |P|=q, \\ 0 & \text{ if } |P|\neq q.\end{cases}
\end{equation}

\begin{proof}[Second proof of Theorem~\ref{thm:MacMahonCrew}]
We claim that for every weighted forest $F=(V,E,\wt)$, we have
\begin{equation}\label{g-from-omega}
g_F(a,b,c,d)=\sum_{\bLambda\vdash(n,w)}\beta_{\bLambda}(F)(-1)^{n-\ell(\bLambda)}\omega(\bLambda,a,b,c,d)
\end{equation}
where
\[\omega(\bLambda,a,b,c,d)=(-1)^{n-a-1}\sum_{\bOmega\partn (b,c)}\binom{c-\ell(\bOmega)}{d} \binom{\bLambda}{\bOmega} \binom{n-\ell(\bLambda)+\ell(\bOmega)-c}{n-a-d-1}.\]
Indeed, let RHS denote the right-hand side of~\eqref{g-from-omega}; then
\begin{align*}
\text{RHS}
&=\sum_{\bLambda\vdash (n,w)}(-1)^{n-\ell(\bLambda)}|\{S\subseteq E:\bitype(S)=\bLambda\}|(-1)^{n-a-1}\sum_{\bOmega\vdash (b,c)}\binom{c-\ell(\bOmega)}{d} \binom{\bLambda}{\bOmega} \binom{n-\ell(\bLambda)+\ell(\bOmega)-c}{n-a-d-1}\\
&=\sum_{S\subseteq E}(-1)^{n-\ell(\bitype(S))}(-1)^{n-a-1}\sum_{\bOmega\vdash (b,c)}\binom{c-\ell(\bOmega)}{d} \binom{\bitype(S)}{\bOmega} \binom{n-\ell(\bitype(S))+\ell(\bOmega)-c}{n-a-d-1}\\
&=\sum_{S\subseteq E}(-1)^{|S|+n-a-1}\sum_{\bOmega\vdash (b,c)}\binom{c-\ell(\bOmega)}{d}\binom{\bitype(S)}{\bOmega} \binom{|S|+\ell(\bOmega)-c}{n-a-d-1}\\
\intertext{(since $S$ is a forest and thus has $n-\ell(\bitype(S))$ edges)}
&=\sum_{S\subseteq E}(-1)^{|S|+n-a-1}\sum_{\bOmega\vdash (b,c)}\sum_{\substack{A\subseteq V\\ |A|=b,\ \wt(A)=c,\ \bitype(A)=\bOmega,\\ S\subseteq E(A)\cup E(\overline{A})}}\binom{c-\ell(\bOmega)}{d} \binom{|S|+\ell(\bOmega)-c}{n-a-d-1}
\intertext{(since $\binom{\bitype(S)}{\bOmega}=|\{A\subseteq V(G):\bitype(A)=\bOmega,\ S\subseteq E(A)\cup E(\overline{A})\}|$)}
&=\sum_{S\subseteq E}(-1)^{|S|+n-a-1}\sum_{\substack{A\subseteq V:\\ |A|=b,\ \wt(A)=c,\\ S\subseteq E(A)\cup E(\overline{A})}} \binom{c-\ell(\bitype(S))}{d} \binom{|S|+\ell(\bitype(S))-c}{n-a-d-1}\\
&=\sum_{\substack{A\subseteq V\\ |A|=b,\ \wt(A)=c}}\sum_{S\subseteq E(A)\cup E(\bar{A})}(-1)^{|S|+n-a-1} \binom{c-\ell(\bitype(S))}{d} \binom{|S|+\ell(\bitype(S))-c}{n-a-d-1}\\
&=\sum_{\substack{A\subseteq V\\ |A|=b,\ \wt(A)=c}}\sum_{S(A)\subseteq E(A)}\sum_{S(\bar{A})\subseteq E(\bar{A})}(-1)^{|S(A)|+|S(\bar{A})|+n-a-1} \binom{|S(A)|}{d} \binom{|S(\bar{A})|}{n-a-d-1}\\
&=\sum_{\substack{A\subseteq V\\ |A|=b,\ \wt(A)=c}}\left(\sum_{S(A)\subseteq E(A)}(-1)^{|S(A)|+d} \binom{|S(A)|}{d}\right)  \left(\sum_{S(\bar{A})\subseteq E(\bar{A})}(-1)^{|S(\bar{A})|+n-a-d-1} \binom{|S(\bar{A})|}{n-a-d-1}\right)\\
&=|\{A\subseteq V: \ext(A)=a,\ |A|=b, \ \wt(A)=c,\ \interior(A)=d \}|
\end{align*}
and applying~\eqref{useful-lemma} yields the left-hand side of~\eqref{g-from-omega}.
\end{proof}

\section{Additional remarks}\label{sec:additional}

\subsection{Chromatic bases of MacMahon symmetric functions} \label{sec:chromatic-bases}

The theory of \textit{chromatic bases of symmetric functions} was introduced in \cite{CvW} (and anticipated in \cite{Scott}).
Let $\mathcal{G}=\{G_n\mid n\in\Pp\}$ be a family of connected graphs such that $G_n$ has $n$ vertices.  Then the chromatic symmetric functions $\CSF_{G_n}$ are algebraically independent, and the free polynomial algebra they generate is the symmetric functions.  To see this, for each partition $\lambda=(\lambda_1,\dots,\lambda_\ell)$, let $G_\lambda$ be the disjoint union $G_{\lambda_1}+\cdots+G_{\lambda_\ell}$.  By Stanley's formula, the matrix that expresses the chromatic symmetric functions of the $G_\lambda$ in the power-sum basis is triangular, hence invertible, so that they form a graded basis for $\Lambda$.  The case that $G_n$ is the star on $n$ vertices has proven especially useful in approaches to Stanley's problem; see, e.g., \cite{APdMOZ,GOT}.

There is an analogous notion of chromatic bases for MacMahon symmetric functions.  The statement is slightly different due to the requirement that weight functions be positive.  Observe that $\Mac^m$ has a Hopf subalgebra $\Mac^m_+$ generated by the MacMahon power-sum functions $p_{\bLambda}$ for which every part of $\bLambda$ is a vector in $\Pp^m$ (as opposed to $\Nn^m\sm\{(0,\dots,0)\}$); this property is preserved by the product and coproduct formulas \eqref{product},~\eqref{coproduct}.  The theory of chromatic bases then carries over with no essential change from symmetric functions to MacMahon symmetric functions:

\begin{proposition}\label{prop:chromatic-basis}
Let $\mathcal{G}=\{G_{n,w} \mid n,w\in\Pp\}$ be a family of connected weighted graphs, where each $G_{n,w}$ has $n$ vertices and total weight $w$.  Then the family $\{\CMF_{G_{n,w}}\}$ generates $\Mac^2_+$ as a free polynomial (Hopf) algebra.
\end{proposition}

\subsection{Multiweighted forests} \label{sec:multiweighted}
Crew's conjecture generalizes easily to weighted graphs whose weights are integer vectors.  Fix $r\in\Pp$, and consider a graph $G=(V,E)$ together with a weight function $\wt\colon V\to\Pp^r$.  Let $C_1,\dots,C_k$ be the connected components of $G$.
In analogy to Definitions~\ref{defn:bitype} and~\ref{defn:CMF}, define the \defterm{weighted type} of $G$ as the vector partition of width~$r+1$ given by
\[\wtype(G)=\big( (n(C_1),\wt(C_1)), \; \dots, \; (n(C_k),\wt(C_k)) \big) \partn (n(G),\wt(G)),\]
and the \defterm{chromatic MacMahon symmetric function} of $G$ as the element of $\Mac^{r+1}$ given by
\[\CMF_G = \sum_{\kappa\in\Col(G)} \prod_{v\in V} x_{\kappa(v)}\prod_{i=1}^r y_{i,\kappa(v)}^{\wt_i(v)}.\]
This power series is an element of $\Mac^{r+1}$, where the alphabet $\{x_1,x_2,\dots\}$ records numbers of vertices and the $r$ alphabets $\{y_{1,1},y_{2,1},\dots\},\;\dots,\;\{y_{1,r},y_{2,r},\dots\}$ record weights.
Finally, analogously to Definition~\ref{defn:EGDP}, define the \textbf{extended generalized degree polynomial} of $G$ as
\[\tilde{\mathbf{G}}_{G} = \sum_{A\subseteq V}w^{\ext(A)}x^{|A|}y_1^{\wt_1(A)}\cdots y_r^{\wt_r(A)}z^{\interior(A)}.\]

With this setup, Proposition~\ref{power-CMF} and Corollary~\ref{power-CMF-forest} (expanding the CMF in the power-sum basis) go through with no changes other than replacing $\bitype(S)$ with $\wtype(S)$.  Likewise, the Hopf-theoretic proof of Theorem~\ref{thm:MacMahonCrew} goes through with little change.  First, the map $\varphi_{t,u,v_1,\dots,v_r}\colon M^{r+1}\rightarrow \mathbb{Q}[t,u,v_1,\dots,v_r]$ defined on the power-sum basis by
\[\varphi_{t,u,v_1,\dots,v_r}(p_{\boldsymbol{\Lambda}}(x,y_1,\dots,y_r))=t^{n}(1-u)^{n-\ell(\boldsymbol{\Lambda})}\prod_{i=1}^rv_i^{w_i}\]
can be shown to satisfy
\[\varphi_{t,u,v_1,\dots,v_r}(\CMF_F)=t^{n(F)}u^{e(F)}\prod_{i=1}^rv_i^{\wt_i(F)}\]
for every forest $F$, and then the map $\gamma\colon M^{r+1}\rightarrow \mathbb{Q}(w,x_1,\dots,x_r,y,z)$ defined by
\[\gamma=\varphi_{wx,w^{-1}z,y_1,\dots,y_r}*\varphi_{w,w^{-1},1,\dots,1}\]
has the property
\[\gamma(\CMF_F)=y^{c(F)}\tilde{\mathbf{G}}_F(w,x_1,\dots,x_r,y,z)\]
thus showing that the extended GDP of a $\Pp^r$-weighted forest can be recovered from its chromatic MacMahon symmetric function.

The theory of chromatic bases also extends easily to the multiweighted setting.  In analogy to Proposition~\ref{prop:chromatic-basis}: for any family $\{G_{n,w_1,\dots,w_r} \mid n,w_1,\dots,w_r\in\Pp\}$ of connected weighted graphs, where each $G_{n,w}$ has $n$ vertices and total weight $(w_1,\dots,w_r)$, their chromatic MacMahon symmetric functions generate $\Mac^{r+1}_+$ as a free polynomial (Hopf) algebra.

\bibliographystyle{amsalpha}
\bibliography{biblio}

\providecommand{\bysame}{\leavevmode\hbox to3em{\hrulefill}\thinspace}
\providecommand{\MR}{\relax\ifhmode\unskip\space\fi MR }
\providecommand{\MRhref}[2]{%
  \href{http://www.ams.org/mathscinet-getitem?mr=#1}{#2}
}
\providecommand{\href}[2]{#2}
\begin{thebibliography}{APdMOZ23}

\bibitem[APdMOZ23]{APdMOZ}
Jos\'e{} Aliste-Prieto, Anna de~Mier, Rosa Orellana, and Jos\'e Zamora,
  \emph{Marked graphs and the chromatic symmetric function}, SIAM J. Discrete
  Math. \textbf{37} (2023), no.~3, 1881--1919. \MR{4632385}

\bibitem[APMWZ24]{AJMWZ}
Jos\'e{} Aliste-Prieto, Jeremy~L. Martin, Jennifer~D. Wagner, and Jos\'e
  Zamora, \emph{Chromatic symmetric functions and polynomial invariants of
  trees}, Bull. Lond. Math. Soc. \textbf{56} (2024), no.~11, 3452--3476.
  \MR{4828026}

\bibitem[CDL94]{CDL}
S.~V. Chmutov, S.~V. Duzhin, and S.~K. Lando, \emph{Vassiliev knot invariants.
  {III}. {F}orest algebra and weighted graphs}, Singularities and bifurcations,
  Adv. Soviet Math., vol.~21, Amer. Math. Soc., Providence, RI, 1994,
  pp.~135--145. \MR{1310599}

\bibitem[Cre20]{CL}
Logan Crew, \emph{Vertex-weighted {G}eneralizations of {C}hromatic {S}ymmetric
  {F}unctions}, ProQuest LLC, Ann Arbor, MI, 2020, Thesis (Ph.D.)--University
  of Pennsylvania. \MR{4106357}

\bibitem[Cre22]{CL2}
\bysame, \emph{A note on distinguishing trees with the chromatic symmetric
  function}, Discrete Math. \textbf{345} (2022), no.~2, Paper No. 112682, 4.
  \MR{4327391}

\bibitem[CS20]{CS}
Logan Crew and Sophie Spirkl, \emph{A deletion-contraction relation for the
  chromatic symmetric function}, European J. Combin. \textbf{89} (2020),
  103143, 20. \MR{4093019}

\bibitem[CvW16]{CvW}
Soojin Cho and Stephanie van Willigenburg, \emph{Chromatic bases for symmetric
  functions}, Electron. J. Combin. \textbf{23} (2016), no.~1, Paper 1.15, 7.
  \MR{3484720}

\bibitem[Die18]{Diestel}
Reinhard Diestel, \emph{Graph {T}heory}, fifth ed., Graduate Texts in
  Mathematics, vol. 173, Springer, Berlin, 2018, Free version available at
  \href{https://diestel-graph-theory.com}{diestel-graph-theory.com}.
  \MR{3822066}

\bibitem[GOT24]{GOT}
Michael Gonzalez, Rosa Orellana, and Mario Tomba, \emph{The chromatic symmetric
  function in the star-basis}, preprint,
  \href{https://arxiv.org/abs/2404.06002}{arXiv:2404.06002}, 2024.

\bibitem[GR14]{GR}
Darij Grinberg and Victor Reiner, \emph{{H}opf algebras in combinatorics},
  preprint, \href{https://doi.org/10.48550/arXiv.1409.8356}{arXiv.1409.8356},
  2014.

\bibitem[LM11]{LM}
Aaron Lauve and Mitja Mastnak, \emph{The primitives and antipode in the {H}opf
  algebra of symmetric functions in noncommuting variables}, Adv. in Appl.
  Math. \textbf{47} (2011), no.~3, 536--544. \MR{2822200}

\bibitem[LS19]{LS}
Martin Loebl and Jean-S\'ebastien Sereni, \emph{Isomorphism of weighted trees
  and {S}tanley's isomorphism conjecture for caterpillars}, Ann. Inst. Henri
  Poincar\'e{} D \textbf{6} (2019), no.~3, 357--384. \MR{4002670}

\bibitem[LT24]{LT}
Ricky~Ini Liu and Michael Tang, \emph{Generalized degree polynomials of trees},
  preprint, \href{https://doi.org/10.48550/arXiv.2411.18972}{arXiv.2411.18972},
  2024.

\bibitem[Mac60]{MacMahon}
Percy~A. MacMahon, \emph{Combinatory analysis}, Chelsea Publishing Co., New
  York, 1960, Two volumes (bound as one). \MR{141605}

\bibitem[MMW08]{MMW}
Jeremy~L. Martin, Matthew Morin, and Jennifer~D. Wagner, \emph{On
  distinguishing trees by their chromatic symmetric functions}, J. Combin.
  Theory Ser. A \textbf{115} (2008), no.~2, 237--253. \MR{2382514}

\bibitem[NW99]{NW}
S.~D. Noble and D.~J.~A. Welsh, \emph{A weighted graph polynomial from
  chromatic invariants of knots}, Ann. Inst. Fourier (Grenoble) \textbf{49}
  (1999), no.~3, 1057--1087. \MR{1703438}

\bibitem[Ros01]{R}
Mercedes~H. Rosas, \emph{Mac{M}ahon symmetric functions, the partition lattice,
  and {Y}oung subgroups}, J. Combin. Theory Ser. A \textbf{96} (2001), no.~2,
  326--340. \MR{1864127}

\bibitem[RRS02]{RRS}
Mercedes~H. Rosas, Gian-Carlo Rota, and Joel Stein, \emph{A combinatorial
  overview of the {H}opf algebra of {M}ac{M}ahon symmetric functions}, Ann.
  Comb. \textbf{6} (2002), no.~2, 195--207. \MR{1955520}

\bibitem[Sco08]{Scott}
Geoffrey Scott, \emph{Characterizing graphs with equal chromatic functions},
  Undergraduate thesis, Dartmouth College, 2008.

\bibitem[Sta95]{S}
Richard~P. Stanley, \emph{A symmetric function generalization of the chromatic
  polynomial of a graph}, Adv. Math. \textbf{111} (1995), no.~1, 166--194.
  \MR{1317387}

\bibitem[Sta99]{EC2}
\bysame, \emph{Enumerative combinatorics. {V}ol. 2}, Cambridge Studies in
  Advanced Mathematics, vol.~62, Cambridge U.\ Press, Cambridge, 1999.
  \MR{1676282}

\bibitem[WYZ24]{WYZ}
Yuzhenni Wang, Xingxing Yu, and Xiao-Dong Zhang, \emph{A class of trees
  determined by their chromatic symmetric functions}, Discrete Math.
  \textbf{347} (2024), no.~9, Paper No. 114096, 11. \MR{4748721}

\end{thebibliography}
\end{document}